\documentclass[journal]{IEEEtran} 
\IEEEoverridecommandlockouts                              

\usepackage{algorithmic,algorithm}
\usepackage{epsfig}
\usepackage{array}
\usepackage{amsmath}
\usepackage{amssymb}
\usepackage{amsthm}
\usepackage{cite}
\usepackage{color}
\usepackage{balance}
\newtheorem{thm}{Theorem}[section]
\newtheorem{lem}[thm]{Lemma}
\newtheorem{cor}[thm]{Corollary}
\newtheorem{ass}{Assumption}
\newtheorem{res}{Restriction}
\newtheorem{rem}{Remark}

\newcommand{\bmat}[1]{\begin{bmatrix}#1\end{bmatrix}}
\newcommand{\xalgm}{\mathcal{X}_{\rm Alg}}
\newcommand{\xalg}{$\xalgm$ }
\newcommand{\xpicker}{${\rm XPicker}(\xalgm, \varepsilon_I)$ }
\newcommand{\onestage}{{\rm OneStage}(\hat{x}; \hat{V}_I)}

\usepackage{xcolor,calc}

\title{Generalized Dual Dynamic Programming for Infinite Horizon Problems in Continuous State and Action Spaces}
\author{Joseph Warrington${}^{\star\dagger}$,~\IEEEmembership{Member,~IEEE,} Paul N.~Beuchat${}^\star$, and John Lygeros${}^\star$,~\IEEEmembership{Fellow,~IEEE}
\thanks{${}^\star$ Automatic Control Laboratory, Swiss Federal Institute of Technology (ETH) Zurich,
Physikstrasse 3, 8092 Zurich, Switzerland. Contact: {\tt\small \{warrington,beuchatp,lygeros\}@control.ee.ethz.ch}}
\thanks{${}^\dagger$ Corresponding author.}}

\begin{document}
\maketitle

\thispagestyle{empty}
\begin{abstract}
We describe a nonlinear generalization of dual dynamic programming theory and its application to value function estimation for deterministic control problems over continuous state and action spaces, in a discrete-time infinite horizon setting. We prove, using a Benders-type argument leveraging the monotonicity of the Bellman operator, that the result of a one-stage policy evaluation can be used to produce nonlinear lower bounds on the optimal value function that are valid over the entire state space. These bounds contain terms reflecting the functional form of the system's costs, dynamics, and constraints. We provide an iterative algorithm that produces successively better approximations of the optimal value function, and prove under certain assumptions that it achieves an arbitrarily low desired Bellman optimality tolerance at pre-selected points in the state space, in a finite number of iterations. We also describe means of certifying the quality of the approximate value function generated. We demonstrate the efficacy of the approach on systems whose dimensions are too large for conventional dynamic programming approaches to be practical.
\end{abstract}


\section{Introduction}

Dynamic Programming (DP) facilitates the selection of optimal actions for multi-stage decision problems, by representing future rewards or costs via a \emph{value function}, or cost-to-go function. This allows the next action to be computed using a smaller ``one-stage'' optimization parameterized by the current system state, without considering the series of future decisions that will follow. This article is concerned with the estimation of value functions for deterministic infinite-horizon problems in discrete time, with continuous state and action (or control input) spaces. This common subclass of Markov decision processes (MDPs) \cite{chow_optimal_1991} is often used to model problems in optimal control \cite{bertsekas_stochastic_2004} and reinforcement learning \cite{busoniu_reinforcement_2010}. 

The canonical DP algorithm was originally developed for problems on discrete state and action spaces \cite{bellman_dynamic_1957}. For continuous state and action spaces it is typical to discretize, or ``grid'' these spaces, and revert to the same algorithmic approach as for discrete problems. This causes two computational issues. Firstly, under quite general assumptions the cost is exponential in the state and action dimensions, $\mathcal{O}([(1-\gamma) \overline{\delta}]^{-(2n + m)})$, where $\gamma$ is the discount factor, $\overline{\delta}$ is the desired value function approximation accuracy, and $n$ and $m$ are the continuous state and action dimensions respectively \cite[eq.~(7.13)]{chow_optimal_1991}. Secondly, to compute an optimal action one must interpolate between grid points of the successor state to estimate the cost-to-go. 

One therefore seeks more practical DP approaches for continuous problems. In rare cases, the most famous being the unconstrained linear-quadratic regulator (LQR) \cite{pappas_numerical_1980}, the optimal value function and control policy can be calculated exactly. The value function for \emph{constrained} LQR is also computable, and is piecewise quadratic \cite{chmielewski_constrained_1996, grieder_computation_2004}. Algorithms exists for computing its polyhedral ``critical regions,'' in each of which the optimal control action is a different affine function of the state \cite{bemporad_explicit_2002}. This computation can also be performed for piecewise linear cost functions. However, the solutions quickly become expensive to compute and store, as the number of regions grows combinatorially with the number of state and input constraints, and with the trajectory length considered. Approximations exist, but can generally be applied only after the full control law has been computed \cite{jones_polytopic_2010}. For almost any other problem, the critical regions are far harder to characterize and compute, and no mature algorithms exist.

A group of methods known  as \emph{approximate DP} (ADP) has arisen for continuous problems where the above methods fail. ADP is closely related to reinforcement learning \cite{sutton_reinforcement_2017}, in that both fields are concerned with finding approximate value functions and policies for MDPs \cite{busoniu_reinforcement_2010} using more or less model information. A number of recent ADP methods relevant to the present article build on the so-called linear programming (LP) approach of de Farias and Van Roy \cite{de_farias_linear_2003}. In this approach one maximizes a candidate value function (over some restricted function class) subject to a so-called \emph{Bellman inequality} constraint, guaranteeing that the result lower-bounds the optimal value function. We critique existing methods in this category after describing our contributions below. 

The second DP method of relevance to this article is \emph{dual DP} (DDP), first used to solve finite-horizon optimization problems in the area of hydro power planning \cite{pereira_multi-stage_1991}, and since applied in finance \cite{dantzig_multi-stage_1993}, economics \cite{egging_benders_2013}, and power system optimization \cite{zhao_multi-stage_2013}. DDP splits a multi-stage problem into single-stage problems, each connected to its successor by an estimate of its value function. The algorithm generates a forward trajectory through the planning horizon, in which control decisions are made by solving the single-stage problems, and then performs a backward pass in which Benders' decomposition \cite{benders_partitioning_1962, geoffrion_generalized_1972}, is used to generate tighter lower-bounding hyperplanes for each stage's value function. At convergence one obtains a locally-exact representation of the value function around the optimal trajectory \cite{philpott_convergence_2008,shapiro_analysis_2011}. 


\subsection{Contributions}

In this article we propose an infinite-horizon, nonlinear generalization of DDP theory, using duality arguments and the properties of the Bellman operator to construct provably-increasing analytical lower bounds on an optimal value function. We refer to this as \emph{generalized DDP} (GDDP). To the best of our knowledge, duality arguments of this kind have not been used to estimate value functions in an infinite-horizon setting before. Specifically, we make the following contributions:

\begin{enumerate}
\item We show, using a Benders decomposition argument, that a globally-valid lower bound on the optimal value function can be parameterized by the dual solution of a single one-stage control problem, in which the ``current'' system state appears as a fixed parameter. The lower bound contains terms reflecting the problem's stage cost, dynamics, and constraints. It is also reflexive, in that it can immediately be used within the same one-stage problem, with the same or a different state parameter, to obtain a new bounding function. This is possible thanks to the monotonicity property of the Bellman operator and the existence of a single value function for the infinite-horizon problem.

\item We provide an iterative algorithm for improving the value function estimate, which is represented as the pointwise maximum of all lower bounds generated so far. Each algorithm iteration has, under certain assumptions on the problem data, polynomial complexity in the state and action dimension. We prove that under a strong duality assumption, an iteration of this algorithm causes strict improvement in the value function estimate at any value of the state where a so-called Bellman error \cite[\S 3.6.1]{busoniu_reinforcement_2010} is present. This is reminiscent of the well-known improvement property of value function iteration for discrete reinforcement learning problems \cite[\S 4.1]{sutton_reinforcement_2017}. 

\item We provide a generic guarantee that the value function estimate converges pointwise over the region of the state space where the optimal value function is finite. If strong duality holds in the one-stage problem, the Bellman error converges to zero for all points in the state space revisited with positive probability at each iteration. The GDDP algorithm we define achieves a given strictly positive Bellman error tolerance at all of a finite number of pre-selected state space points, within finite iterations.

\item We apply GDDP to linear and nonlinear control problems for which conventional direct DP approaches fail.
\end{enumerate}

GDDP offers several potential attractions. Each lower-bounding function generalizes an estimate of the value function away from the single point in the state space where it was derived, efficiently exploiting the dual result of the one-stage problem. In contrast, ``gridding'' approaches to DP for continuous-state problems do not generate global lower bounds at all, and have an unattractive up-front exponential cost \cite{chow_optimal_1991}. Although we do not claim to have overcome the so-called curse of dimensionality in this manner -- in particular we do not bound the number of iterations required -- the algorithm at least avoids the need to interpolate between points on a grid \cite{chow_optimal_1991}, or in the case of local approximation techniques, to identify approximate value functions around nearby sampled points \cite{atkeson_random_2008}.\footnote{In some circumstances with a discrete action set, random sampling of the state can lead to a DP algorithm for which the solution time remains polynomial in the state dimension \cite{rust_using_1997}. However to our knowledge there are no such results for problems with continuous state \emph{and} action spaces.} 

GDDP can be compared with the LP approach to ADP, as both methods maximize a value function while respecting a Bellman inequality condition. Our pointwise maximum representation of the value function is more expressive than the single polynomial employed in \cite{summers_approximate_2013}, and as such we do not rely on high-order parameterizations to obtain a good-quality approximation. It is also clearly more expressive than the earlier quadratic bound \cite{wang_performance_2009}. Pointwise maximum representations exist within the LP-approach literature \cite{wang_approximate_2015, beuchat_point-wise_2017}. In \cite{wang_approximate_2015} an \emph{iterated} Bellman inequality is derived, enlarging the set of feasible parameterizations. However in that formulation one must optimize over all value function ``iterations'' simultaneously, and the option to take a pointwise maximum is essentially a by-product of the approach. The authors of \cite{beuchat_point-wise_2017} avoid this simultaneous optimization by converting the result of each iteration into fixed input data for the next. However all of \cite{summers_approximate_2013}, \cite{wang_approximate_2015}, and \cite{beuchat_point-wise_2017} are difficult to extend beyond polynomial dynamics due to the way they reformulate the Bellman inequality constraint. Another example of an approach accommodating a pointwise maximum representation is \cite{lincoln_relaxing_2006}, in which the Bellman condition is relaxed by a predetermined multiplicative factor. Inner- and outer-approximate value functions are then constructed to satisfy the relaxed condition. Although the algorithmic principle is general, the parameterization of these bounding functions requires application-specific insight. Our Benders argument is distinct from existing ADP literature, leads to implementable algorithms for a wide range of problems, and requires no \textit{a priori} knowledge of the form of the optimal value function. For linear systems, GDDP involves solving more scalable optimization problems than the semidefinite programs of \cite{summers_approximate_2013}, \cite{wang_approximate_2015}, and \cite{beuchat_point-wise_2017}.

\subsection{Article structure}

Section \ref{sec:ProblemStatement} states the infinite-horizon control problem to be studied, and reviews standard results on value functions and the Bellman equation. Section \ref{sec:LBAlgorithm} derives the proposed algorithm for producing successive approximations to the optimal value function, and proves some of its key properties. Section \ref{sec:Implementation} discusses implementation choices. Section \ref{sec:Results} presents numerical simulations for linear systems of various sizes as well as a nonlinear example, and Section \ref{sec:Conclusion} concludes.

\subsection{Notation}

The symbol $\mathbb{R}^n$ ($\mathbb{R}^n_+$) represents the space of (non-negative) $n$-dimensional vectors, and $\mathbb{S}^n_{++}$ ($\mathbb{S}^n_{+}$) represents the cone of symmetric positive (semi-)definite $n \times n$ matrices. Inequality $a \leq b$ for $n$-dimensional vectors means $b - a \in \mathbb{R}^n_+$, $A \preceq B$ for $n \times n$ matrices means $B - A \in \mathbb{S}_+^n$, and $A \prec B$ means $B - A \in \mathbb{S}_{++}^n$. The spectral radius (magnitude of the largest-magnitude eigenvalue) of a square matrix $A$ is denoted $\sigma(A)$. Notation $\max \{ a, b, ... \}$ represents the maximum of scalars $a$, $b$, and so on. When $\max\{\ldots\}$ is used with a subscript parameter below the symbol $\max$, it denotes the maximum attained value of the parameterized quantity in the braces. The symbol $0$ represents a vector or matrix of zeroes of appropriate size, and $\mathbf{1}$ represents a vector of ones. Symbol $I_n$ denotes the $n \times n$ identity matrix.  The notation ${\rm diag}\{a, b, \ldots\}$ signifies a diagonal matrix whose entries are $a$, $b$, and so on.


\section{Problem statement} \label{sec:ProblemStatement}

\subsection{Infinite-horizon control problem} \label{sec:IHProblem}

We consider the solution of an infinite-horizon deterministic optimal control problem in which the stage cost function, dynamics, and constraints are time-invariant:
\begin{subequations} \label{eq:IHProblem}
\begin{align}
V^\star(x) := \inf_{u_0,u_1,\ldots} \quad & \sum_{t=0}^{\infty} \gamma^t \ell(x_t, u_t) \label{eq:IHObj}\\
\text{s.~t.} \quad & x_{t+1} = f(x_t, u_t), \quad t = 0,1,\ldots \, , \\
& Eu_t \leq h(x_t), \quad t = 0,1,\ldots \, , \label{eq:IHSIC} \\
& x_0 = x \, .
\end{align}
\end{subequations}
The state at time $t$ is denoted $x_t \in \mathcal{X} \subseteq \mathbb{R}^n$, and the action, or input, is denoted $u_t \in \mathcal{U} \subseteq \mathbb{R}^m$. Constant $\gamma \in \left(0, 1\right]$ is a discount factor for costs encountered later in the horizon, $\ell : \mathcal{X} \times \mathcal{U} \rightarrow \mathbb{R}$ is the stage cost function, and the dynamics are characterized by $f : \mathcal{X} \times \mathcal{U} \rightarrow \mathcal{X}$. The sets $\mathcal{X}$ and $\mathcal{U}$ are the continuous state and action spaces respectively. The state-input constraints \eqref{eq:IHSIC} are parameterized by $E \in \mathbb{R}^{n_c \times m}$, where $n_c$ is the number of such constraints present, and a mapping $h : \mathcal{X} \rightarrow \mathbb{R}^{n_c}$ from the state to a vector of right-hand-side quantities. The function $V^\star(x)$ is the parametric infimum of problem \eqref{eq:IHProblem} as $x$ is varied.


\subsection{Optimal value function}

We define the Bellman operator $\mathcal{T}$ on a proper function $V: \mathcal{X} \rightarrow \mathbb{R} \cup \{+\infty\}$ in terms of the optimal objective value of a \emph{one-stage problem} related to \eqref{eq:IHProblem}:
\begin{equation} \label{eq:TVDef}
\mathcal{T}V(x) := \inf_{u \in \mathcal{U}(x)} \left\{ \ell(x, u) + \gamma V(f(x, u)) \right\} \, ,
\end{equation}
where 
\begin{equation}
\mathcal{U}(x) := \left\{ u \in \mathcal{U} \, :\,  Eu \leq h(x) \right\} \, .
\end{equation}
We refer to function $V$ as an \emph{approximate value function}, as its purpose in \eqref{eq:TVDef} is to represent all costs incurred from time $t+1$ onwards along the trajectory. Where the infimum in \eqref{eq:TVDef} is attained, one can define a control policy satisfying
\begin{equation} \label{eq:PiDef}
u = \pi(x) \in \underset{u \in \mathcal{U}(x)}{\arg \min}\left\{ \ell(x, u) + \gamma V(f(x, u)) \right\} \, .
\end{equation}
In accordance with other DP literature \cite[\S 2.2]{busoniu_reinforcement_2010}, we refer to $\pi(x)$ as the \emph{greedy policy} for $V$, since only the stage cost for the immediate time step is modelled explicitly in the calculation. 

An \emph{optimal value function} $V^\star$ for problem \eqref{eq:IHProblem} satisfies the well-known Bellman optimality condition,
\begin{equation} \label{eq:BellOpt}
V^\star(x) = \mathcal{T}V^\star(x) \, , \quad \forall x \in \mathcal{X} \, .
\end{equation}
We make the following assumption:
\begin{ass} \label{ass:UniqueVStar}
The stage cost function $\ell(\cdot,\cdot)$ is non-negative on its domain, and a unique, time-invariant, optimal value function satisfying \eqref{eq:BellOpt} exists.
\end{ass}

This existence condition is satisfied in various settings. For example, a straightforward guarantee of a unique $V^\star$ is that the stage cost be bounded on $\mathcal{X}$, and the discount factor $\gamma$ be strictly less than $1$. If in addition $\mathcal{U}(x)$ is non-empty for all $x \in \mathcal{X}$, then $V^\star$ is also finite on $\mathcal{X}$ \cite[Prop.~1.2.3]{bertsekas_dynamic_2012}.  Alternatively, unbounded stage costs such as quadratic functions over a non-compact domain may be accommodated under different assumptions, such as those of Hern{\' a}ndez-Lerma and Lasserre \cite[\S 4.2]{hernandez-lerma_discrete-time_1996}. Our later results are generic enough to rely only on Assumption \ref{ass:UniqueVStar} itself holding, and not on specific settings such as these. 

An additional, obvious property of $V^\star$ follows:

\begin{lem}[Non-negative optimal value function] \label{lem:VGEZero}
Under Assumption \ref{ass:UniqueVStar}, $V^\star(x) \geq 0 \, \forall x \in \mathcal{X}$.
\begin{proof}
The stage costs summed in \eqref{eq:IHProblem} are non-negative, thus their infimum over control decisions is also non-negative.
\end{proof}
\end{lem}
If the limit under repeated value iteration $\lim_{N \rightarrow \infty}$ $\mathcal{T}^N V(x)$ exists at $x$, then $V^\star(x)$ is finite and equal to this limiting value \cite[Prop.~1.2.1]{bertsekas_dynamic_2012}. Otherwise, we say $V^\star(x) = +\infty$ and that $V^\star(x) = \mathcal{T}V^\star(x)$ there by convention.


\subsection{Restriction of problem class}

In addition to assuming that the stage cost is non-negative, we restrict it to the following form, which ensures it will be compatible with our use of epigraph variables in Section \ref{sec:LBAlgorithm}.
\begin{res} \label{res:StageCost}
The stage cost $\ell(x,u)$ consists of $K \geq 1$ terms, each of which is the pointwise maximum of general functions of $x$ plus linear and quadratic terms in $u$:
\begin{equation}  \label{eq:ResStageCost}
\ell(x,u) = \sum_{k=1}^K \ell_k(x,u) \, ,
\end{equation}
where
\begin{equation} \label{eq:ResStageCost2}
\ell_k(x,u) := \max_l \left\{ \phi_{kl}(x) +  r_{kl}^\top u + \frac{1}{2} u^\top R_{kl} u  \right\}
\end{equation}
and the maximum is over a finite number of indices $l$ for each $k$. Each function $\phi_{kl}(x)$ is finite-valued for any given $x \in \mathcal{X}$.
\end{res}

Restriction \ref{res:StageCost} accommodates, but is not limited to, affinely-scaled $p$-norms for $p \in \{1,2,\infty\}$, and convex piecewise affine functions of $x$ and $u$.

In fact, to aid developments from Section \ref{sec:LBAlgorithm} onward, we will model the stage cost as the sum of $K$ epigraph variables, $\ell(x,u) = \mathbf{1}^\top \beta$ for $\beta \in \mathbb{R}^K$, and replace the $K$ terms of the form \eqref{eq:ResStageCost2} with a single list of $J$ constraints on $\beta$, \[ e_j^\top \beta \geq \phi_j(x) + r_j^\top u + \frac{1}{2}u^\top R_j u \, , \quad j = 1,\ldots,J \, . \] All indices $kl$ have been rolled into a single index $j$, and $e_j$ is a unit vector selecting the element of $\beta$ to which the $j^\text{th}$ constraint applies.

We also restrict the dynamics:
\begin{res} \label{res:Dynamics}
The dynamics $f(x,u)$ take one of the following input-affine forms:
\begin{enumerate} 
\item[(a)] If in \eqref{eq:ResStageCost2} we have $R_{kl} \succ 0$ for each index $kl$, then $f(x,u) = f_x(x) + F_u(x)u$, where $F_u(x)$ has finite entries for any given $x \in \mathcal{X}$.
\item[(b)] Otherwise, the dynamics take the form $f(x,u) = f_x(x) + Bu$, where $B \in \mathbb{R}^{n \times m}$.
\end{enumerate}
In both (a) and (b), $f_x(x)$ has finite entries for any given $x \in \mathcal{X}$.
\end{res}
The reasons for Restriction \ref{res:Dynamics} will become clear in the derivation of lower bounds on $V^\star(x)$ in Section \ref{sec:LBLemma}.

\subsection{Unconstrained LQR} \label{sec:LQR}

A special tractable case of problem \eqref{eq:IHProblem} is the unconstrained LQR problem, in which $f(x,u) = Ax + Bu$ with $(A,B)$ stabilizable, and $\ell(x,u) = \tfrac{1}{2}x^\top Q x + \tfrac{1}{2}u^\top R u$ with $Q \succeq 0$ and $R \succ 0$. For this problem the optimal value function is $V^\star(x) = \frac{1}{2}x^\top P x$, where $P$ is the solution of the (discounted) discrete algebraic Riccati equation \cite{pappas_numerical_1980}. The discount factor $\gamma$ is handled via the substitutions $\tilde{A} := \sqrt{\gamma}A$ and $\tilde{B} := \sqrt{\gamma}B$.


\section{Lower bounding algorithm} \label{sec:LBAlgorithm}

We now describe an algorithm for producing successively better approximations of $V^\star(x)$ from below.

\subsection{Epigraph form of one-stage problem \eqref{eq:TVDef}}

Let an approximation of the optimal value function be denoted $\hat{V}_I$, taking the form
\begin{equation} \label{eq:VHatDef}
\hat{V}_I(x) = \max_{i=0,\ldots,I} g_i(x)
\end{equation}
where $g_i : \mathcal{X} \rightarrow \mathbb{R} \cup \{+\infty\}$ are proper functions and $I$ is finite. We define $g_0(x) = 0$, which from Lemma \ref{lem:VGEZero} is always a valid lower bound on $V^\star(x)$.

Under Restrictions \ref{res:StageCost} and \ref{res:Dynamics} the one-stage problem in \eqref{eq:TVDef} can be rewritten using epigraph variables $\beta \in \mathbb{R}^{K}$ and $\alpha \in \mathbb{R}$ to represent the stage cost and cost-to-go respectively. Using dynamics of form (a) from Restriction \ref{res:Dynamics} -- the derivation is the same for form (b) -- one obtains the following:
\begin{subequations}\label{eq:OSP}
\begin{align}
\inf_{u, x_+, \beta, \alpha} \quad & \mathbf{1}^\top \beta + \gamma \alpha \\
\text{s.~t.} \quad & x_+ = f_x(\hat{x}) + F_u(\hat{x})u\, , \label{eq:OSPDyn} \\
& Eu \leq h(\hat{x}) \, , \label{eq:OSPConstr}\\
& e_j^\top \beta \geq \phi_j(\hat{x}) + r_j^\top u + \tfrac{1}{2}u^\top R_j u\,, \nonumber \\
& \hspace{3.9cm} j = 1,\ldots,J\, , \label{eq:OSPStageCost}\\
& \alpha \geq g_i(x_+)\,, \quad i = 0\,, \ldots, I\, . \label{eq:OSPLB}
\end{align}
\end{subequations}
This is a parametric optimization problem in which the current state of the system $\hat{x}$ is the (fixed) parameter. Vector $x_+$ has been introduced to model explicitly the successor state to $\hat{x}$ when input $u$ is applied.

The following lemma establishes equivalence between \eqref{eq:OSP} and \eqref{eq:TVDef}.
 
\begin{lem} \label{lem:OSPTV}
If problem \eqref{eq:OSP} is feasible and an optimal solution $(\hat{u}^\star, \hat{x}_+^\star, \hat{\beta}^\star, \hat{\alpha}^\star)$ is attained for parameter $\hat{x}$, \[\mathbf{1}^\top \hat{\beta}^\star + \gamma\hat{\alpha}^\star = \min_{u \in \mathcal{U}(\hat{x})} \left\{ \ell(\hat{x}, u) + \gamma \hat{V}_I(f(\hat{x},u)) \right\} = \mathcal{T}\hat{V}(\hat{x}) \, . \]
\end{lem}
\begin{proof}

The condition $u \in \mathcal{U}(\hat{x})$ is reflected in constraint \eqref{eq:OSPConstr}, and constraint \eqref{eq:OSPDyn} can be eliminated by substituting the definition of $x_+$ into constraint \eqref{eq:OSPLB}. It must hold that $\mathbf{1}^\top \hat{\beta}^\star = \sum_{k=1}^K \hat{\beta}^\star_k = \ell(\hat{x},\hat{u}^\star)$, and that $\gamma\hat{\alpha}^\star = \gamma\hat{V}_I(f(\hat{x}, \hat{u}^\star))$, otherwise the epigraph constraints \eqref{eq:OSPStageCost} and \eqref{eq:OSPLB} are either violated, or are not binding.
\end{proof}


\subsection{Lower-bounding lemma} \label{sec:LBLemma}

We now derive the dual of problem \eqref{eq:OSP} in order to support the lemma that follows. Assign the multipliers $\nu \in \mathbb{R}^n$ to constraint \eqref{eq:OSPDyn}, $\lambda_c \in \mathbb{R}_+^{n_c}$ to constraint \eqref{eq:OSPConstr}, $\lambda_\beta \in \mathbb{R}_+^J$ to constraint \eqref{eq:OSPStageCost}, and $\lambda_\alpha \in \mathbb{R}_+^{I+1}$ to constraint \eqref{eq:OSPLB}. We use $\lambda_{\beta, j}$ to denote the $j^\text{th}$ element of $\lambda_\beta$, and $\lambda_{\alpha, i}$ to denote the $i^\text{th}$ element of $\lambda_\alpha$. The Lagrangian of \eqref{eq:OSP} is, after grouping terms,
{\allowdisplaybreaks
\begin{align}\label{eq:OSPLag}
& \mathcal{L}(u,x_+,\beta,\alpha,\nu,\lambda_c,\lambda_\beta,\lambda_\alpha) := \nonumber \\
 & \, (\mathbf{1} - L^\top \lambda_\beta)^\top \beta + (\gamma - \mathbf{1}^\top \lambda_\alpha)\alpha - \nu^\top x_+ + \sum_{i=0}^I \lambda_{\alpha,i} g_i(x_+) \nonumber \\
& + (\nu^\top F_u(\hat{x}) + \lambda_c^\top E + \lambda_\beta^\top \overline{R})u + \frac{1}{2}u^\top \left( \sum_{j=1}^J \lambda_{\beta, j} R_j \right)u \nonumber \\
& + \nu^\top f_x(\hat{x}) - \lambda_c^\top h(\hat{x}) + \lambda_\beta^\top \phi(\hat{x}) \, ,
\end{align}}
where for compactness we have introduced the additional symbols $L \in \mathbb{R}^{J \times K}$, $\overline{R} \in \mathbb{R}^{J \times m}$ and $\phi(\hat{x}) \in \mathbb{R}^J$, whose $j^\text{th}$ rows contain $e_j^\top$, $r_j^\top$, and $\phi_j(\hat{x})$ respectively. The dual of problem \eqref{eq:OSP} is:
\begin{subequations}\label{eq:OSPD}
\begin{align}
\sup_{\nu, \lambda_c, \lambda_\beta, \lambda_\alpha} \quad & \phi(\hat{x})^\top \lambda_\beta - h(\hat{x})^\top \lambda_c + f_x(\hat{x})^\top\nu \nonumber \\ & \hspace{1.5cm} + \zeta_1(\nu, \lambda_\alpha) + \zeta_2(\hat{x}, \nu, \lambda_c, \lambda_\beta) \label{eq:OSPDObj}\\
\text{s.~t.}\quad &  L^\top \lambda_\beta = \mathbf{1} \, ,\\
& \mathbf{1}^\top \lambda_\alpha = \gamma \, ,\\
& \lambda_c \geq 0, \,\, \lambda_\beta \geq 0, \,\, \lambda_\alpha \geq 0 \, ,
\end{align}
\end{subequations}
where 
\begin{equation}
\zeta_1(\nu, \lambda_\alpha) :=  \inf_{x_+ \in \mathcal{X}} \left\{-\nu^\top x_+ + \sum_{i=0}^I \lambda_{\alpha,i} g_i(x_+) \right\} \label{eq:Delta1Def} 
\end{equation}
and
\begin{align}
& \zeta_2(\hat{x}, \nu, \lambda_c, \lambda_\beta) := \nonumber \\ & \inf_{u \in \mathcal{U}} \! \left\{ \!\left(\nu^\top F_u(\hat{x}) + \lambda_c^\top E + \lambda_\beta^\top \overline{R}\right)u + \frac{1}{2}u^\top \!\! \left( \sum_{j=1}^J \lambda_{\beta, j} R_j \!\right) \!\!u \! \right\}\!. \label{eq:Delta2Def}
\end{align}
Terms $\zeta_1$ and $\zeta_2$ appear in the dual objective as a standard consequence of minimizing the Lagrangian over primal variables. The requirement for the two minimizations to be bounded from below (which ensures the dual function is meaningful) may place additional implicit constraints on the dual variables in \eqref{eq:OSPD}. We do not write these constraints out explicitly, because they do not always apply and their form is not in fact relevant. 

It will, however, be crucial for the Benders-type argument of Lemma \ref{lem:NewLB} below that any extra implicit constraints in \eqref{eq:OSPD} are invariant to $\hat{x}$. Thus \eqref{eq:Delta2Def} could be problematic if the minimization is unbounded for some $\hat{x}$. If all matrices $R_j$ are strictly positive definite then $\zeta_2(\hat{x},\nu,\lambda_c,\lambda_\beta)$ has a finite analytic value for any $\hat{x}$, and no extra constraint coupling the variables $\hat{x}$, $\nu$, $\lambda_c$, and $\lambda_\beta$ is needed. If there is an $R_j$ which is \emph{not} strictly positive definite, then one needs $F_u(\hat{x})^\top \nu + E^\top\lambda_c + \overline{R}^\top \lambda_\beta$ to lie in the span of the eigenvectors of $\sum_{j=1}^J \lambda_{\beta, j} R_j$ with strictly positive eigenvalues. This span constraint must then be invariant to $\hat{x}$, which is satisfied if $F_u(\hat{x})$ is a constant. These issues justify Restriction \ref{res:Dynamics}, which is sufficient to ensure that constraint invariance holds.

Let $J_P(\hat{x})$ denote the optimal objective value of problem \eqref{eq:OSP} as a function of the parameter $\hat{x}$, and similarly define $J_D(\hat{x})$ as the optimal objective value of problem \eqref{eq:OSPD}. We now state the lemma on which our proposed algorithm is based.

\begin{lem}[Lower-bounding lemma] \label{lem:NewLB}
Suppose $g_i(x) \leq V^\star(x), \, \forall x \in \mathcal{X}$, for $i=0,\ldots,I$. Assume that optimal dual variables exist for problem \eqref{eq:OSPD} with parameter $\hat{x}$, and denote these $(\hat{\nu}^\star, \hat{\lambda}_c^\star, \hat{\lambda}_\beta^\star, \hat{\lambda}_\alpha^\star)$. Then the following relationship holds:
\begin{align} \label{eq:gIplus1}
g_{I+1}(x) & := \hat{\lambda}^{\star\top}_\beta \phi(x)  - \hat{\lambda}^{\star\top}_c h(x) + \hat{\nu}^{\star\top}f_x(x) \nonumber \\
& \hspace{1.5cm} + \zeta_1(\hat{\nu}^\star, \hat{\lambda}^\star_\alpha) + \zeta_2(x, \hat{\nu}^\star, \hat{\lambda}^\star_c, \hat{\lambda}^\star_\beta) \nonumber \\ & \leq V^\star(x) \, , \, \forall x \in \mathcal{X} \, .
\end{align}
\end{lem}
\begin{proof}
Consider solving problem \eqref{eq:OSPD} for some other $\overline{x} \neq \hat{x}$. If the supremum is attained there, let \mbox{$(\overline{\nu}^\star, \overline{\lambda}_c^\star, \overline{\lambda}_\beta^\star, \overline{\lambda}_\alpha^\star)$} denote an optimal solution. Then
\begin{subequations}
\begin{align}
J_D(\overline{x}) & = \phi(\overline{x})^\top \overline{\lambda}^\star_\beta  - h(\overline{x})^\top \overline{\lambda}^\star_c + f_x(\overline{x})^\top\overline{\nu}^\star \nonumber \\
& \qquad \qquad + \zeta_1(\overline{\nu}^\star, \overline{\lambda}^\star_\alpha) + \zeta_2(\overline{x}, \overline{\nu}^\star, \overline{\lambda}^\star_c, \overline{\lambda}^\star_\beta) \label{eq:JDxOpt} \\
& \geq \phi(\overline{x})^\top \hat{\lambda}^\star_\beta - h(\overline{x})^\top \hat{\lambda}^\star_c + f_x(\overline{x})^\top\hat{\nu}^\star \nonumber \\
& \qquad \qquad + \zeta_1(\hat{\nu}^\star, \hat{\lambda}^\star_\alpha) + \zeta_2(\overline{x}, \hat{\nu}^\star, \hat{\lambda}^\star_c, \hat{\lambda}^\star_\beta) \, , \label{eq:JDxLB}
\end{align}
\end{subequations}
where line \eqref{eq:JDxOpt} is simply the definition of an optimal dual solution, and \eqref{eq:JDxLB} follows trivially as $(\hat{\nu}^\star, \hat{\lambda}_c^\star, \hat{\lambda}_\beta^\star, \hat{\lambda}_\alpha^\star)$ cannot be better than any optimal solution. Under Restriction \ref{res:Dynamics} the feasible sets of problem \eqref{eq:OSPD} are the same for parameters $\hat{x}$ and $\overline{x}$, and therefore this inequality always holds. 

Even if the supremum in \eqref{eq:OSPD} for parameter $\overline{x}$ is not attained, or the optimal objective value is infinite, i.e.~$J_D(\overline{x}) = +\infty$, then the inequality between $J_D(\overline{x})$ and \eqref{eq:JDxLB} still holds.

Now for any $\overline{x}$, the following relationships also hold:
\begin{equation} \label{eq:Lemma1BellmanSeries}
J_P(\overline{x}) = \mathcal{T}\hat{V}_I(\overline{x}) \leq \mathcal{T}V^\star(\overline{x}) = V^\star(\overline{x}) \, ,
\end{equation}
where the left-hand equality arises from Lemma \ref{lem:OSPTV}, the central inequality comes from monotonicity of the Bellman operator \cite[Lemma 1.1.1]{bertsekas_dynamic_2012} and the fact that $\hat{V}_I(\overline{x}) \leq V^\star(\overline{x})$, and the right-hand equality comes from the Bellman optimality condition \eqref{eq:BellOpt}. 

Furthermore, from weak duality, $J_D(\overline{x}) \leq J_P(\overline{x})$. Therefore, referring to any $\overline{x}$ simply as $x$, and combining \eqref{eq:JDxLB}, \eqref{eq:JDxOpt}, weak duality, and \eqref{eq:Lemma1BellmanSeries}, the result follows. Note that definition \eqref{eq:gIplus1} is the same as \eqref{eq:JDxLB}; we have transposed the first three terms as the fixed Lagrange multipliers appear as coefficients for functions of $x$ in the algorithm.
\end{proof}

\begin{rem}[Alternative form of $g_{I+1}(x)$]
An equivalent definition of $g_{I+1}(x)$, using $J_D(\hat{x})$ and difference terms between $x$ and $\hat{x}$, is
\begin{align} \label{eq:gIplus1Alt}
g_{I+1}(x) & = J_D(\hat{x}) + \hat{\lambda}^{\star\top}_\beta(\phi(x) - \phi(\hat{x})) - \hat{\lambda}^{\star\top}_c(h(x) - h(\hat{x})) \nonumber \\
& \qquad\qquad + \hat{\nu}^{\star\top} (f_x(x) - f_x(\hat{x})) \nonumber \\
& \qquad\qquad + \zeta_2(x, \hat{\nu}^\star, \hat{\lambda}^\star_c, \hat{\lambda}^\star_\beta) - \zeta_2(\hat{x}, \hat{\nu}^\star, \hat{\lambda}^\star_c, \hat{\lambda}^\star_\beta) \, .
\end{align}
In case strong duality holds between problems \eqref{eq:OSP} and \eqref{eq:OSPD}, $J_P(\hat{x})$ can be used in place of $J_D(\hat{x})$, and the new function can be expressed in terms of the optimal objective value and Lagrange multipliers from problem \eqref{eq:OSP}.
\end{rem}

\subsection{Generalized DDP algorithm}

Lemma \ref{lem:NewLB} states that if functions $g_i(x)$ for $i=0,\ldots,I$ are all known to be global lower bounds on $V^\star(x)$, then a new lower bound $g_{I+1}(x)$ can be generated from the solution of dual problem \eqref{eq:OSPD} with parameter $\hat{x}$, as long as that problem attains a finite optimum. Algorithm \ref{alg:GDDP} constructs a series of valid lower-bounding functions based on this result, starting with $g_0(x) = 0$, by solving the one-stage problem at multiple pre-selected points in the state space $\mathcal{X}$, the set of which we denote $\xalgm := \{x_1, \ldots, x_M\}$. Defining the \textit{Bellman error} $\varepsilon(x; V)$ for any approximate value function $V$ and state $x$ as 
\begin{equation} \label{eq:EpsDef}
\varepsilon(x;V) :=  \mathcal{T}V(x) - V(x) \, ,
\end{equation}
the algorithm systematically approaches the condition $\varepsilon(x_m, \hat{V}_I) = 0$ for all $x_m \in \xalgm$. Denoting $\varepsilon_I \in \mathbb{R}^M$ the vector of such Bellman errors at iteration $I$, the algorithm terminates when $||\epsilon_I||_\infty \leq \delta$, and outputs a final approximation $\hat{V}(x)$.\footnote{For the purpose of assessing convergence, we let $\varepsilon(x_m, \hat{V}_I) = 0$ by convention for any $x_m$ where the one-stage problem is infeasible.}

In the algorithm listing, the function \xpicker chooses an element of \xalg to use as the parameter $\hat{x}$ when solving the one-stage problem \eqref{eq:OSP} (or its explicit dual \eqref{eq:OSPD}). Strategies for selecting $\hat{x}$ are described in Section \ref{sec:ChoiceXBeq}. The function $\onestage$ performs this optimization and returns, if optimal dual variables are available, the new lower-bounding function $g_{I+1}(x)$ as described in Lemma \ref{lem:NewLB}. 

\begin{algorithm}[t]
\caption{Generalized Dual Dynamic Programming algorithm for $M$ fixed state space points} 
\label{alg:GDDP}
\begin{algorithmic}[1]
\STATE Generate samples $\xalgm := \{x_1, \ldots, x_M\}$
\STATE Set $I = 0$
\STATE Set $g_0(x) = 0$
\WHILE{TRUE}
\STATE $\hat{V}_I(x) \gets \max_{i=0,\ldots,I} g_i(x)$
\FOR{$x_m \in \xalgm$} \label{algl:BellGapStart}
\STATE $\varepsilon(x_m; \hat{V}_I) \gets \mathcal{T}\hat{V}_I(x_m) - \hat{V}_I(x_m)$
\ENDFOR \label{algl:BellGapEnd}
\IF{$||\varepsilon_I||_\infty \leq \delta$} \label{algl:ConvCheck}
\STATE \textbf{break}
\ENDIF
\STATE $\hat{x} \gets$ \xpicker 
\IF{$\onestage$ is feasible}
\STATE $g_{I+1}(x) \gets \onestage$ according to \eqref{eq:gIplus1} \label{algl:AddLB}
\ELSE
\STATE $V^\star(\hat{x}) = +\infty$; do not revisit $\hat{x}$
\ENDIF
\STATE $I \gets I+1$
\ENDWHILE
\STATE Output $\hat{V}(x) := \max_{i=0,\ldots,I} g_i(x)$
\end{algorithmic}
\end{algorithm}

We now prove some key properties of Algorithm \ref{alg:GDDP}.

\begin{lem}[Positive Bellman error] \label{lem:VlessthanTV}
For the value function under-approximator $\hat{V}_I(x)$ generated by Algorithm \ref{alg:GDDP} at each iteration $I$, 
\begin{equation} \label{eq:VlessthanTV}
\varepsilon(x;\hat{V}_I) \geq 0 \, , \quad \forall x \in \mathcal{X} \, .
\end{equation}
\end{lem}
\begin{proof}
For $I = 0$, the result holds trivially since $\hat{V}_0(x) = g_0(x) = 0$, and $\mathcal{T}\hat{V}_0(x)$ is non-negative for any $x$ by virtue of non-negative stage costs $\ell(x,u)$. 

For any $I > 0$ we have, from definition \eqref{eq:VHatDef}, $\hat{V}_{I}(x) \geq \hat{V}_{I-1}(x)$ for all $x \in \mathcal{X}$, and therefore, by monotonicity of the Bellman operator, $\mathcal{T}\hat{V}_{I}(x) \geq \mathcal{T}\hat{V}_{I-1}(x)$ for all $x \in \mathcal{X}$. Applying the proof of Lemma \ref{lem:NewLB} to iteration $I-k$, for any $k=0,\ldots,I-1$, we have \[g_{I-k}(x) \leq J_D(x) \leq J_P(x) = \mathcal{T}\hat{V}_{I-k-1}(x), \quad \forall x \in \mathcal{X} \, .\] Since $\mathcal{T}\hat{V}_{I-k-1}(x) \leq \mathcal{T}\hat{V}_{I-k}(x) \leq \ldots \leq \mathcal{T}\hat{V}_I(x)$ we have \[ \mathcal{T}\hat{V}_I(x) \geq g_i(x)\, , \quad \forall x \in \mathcal{X}, \,\, \forall i = 0,\ldots,I,\]
from which the result is immediate by the definitions of $\hat{V}_I(x)$ and $\varepsilon(x;\hat{V}_I)$.
\end{proof}

\begin{lem}[Strict increase in value function approximator]\label{lem:VFIncrease}
Suppose $\hat{V}_I(\hat{x}) < \mathcal{T}\hat{V}_I(\hat{x})$ for some $\hat{x} \in \mathcal{X}$, and that $\hat{x}$ is chosen as the evaluation point in iteration $I+1$ of Algorithm \ref{alg:GDDP}. If strong duality holds between problems \eqref{eq:OSP} and \eqref{eq:OSPD}, then this iteration brings about a strict increase in the value function approximation at $\hat{x}$, i.e., $\hat{V}_{I+1}(\hat{x}) > \hat{V}_I(\hat{x})$.  The increase is equal to $\mathcal{T}\hat{V}_I(\hat{x}) - \hat{V}_I(\hat{x})$, i.e.,
\begin{equation} \label{eq:LemStrictIncr}
\hat{V}_{I+1}(\hat{x}) - \hat{V}_I(\hat{x}) = \mathcal{T}\hat{V}_I(\hat{x}) - \hat{V}_I(\hat{x}) > 0 \, .
\end{equation}
\end{lem}
\begin{proof}
From definition \eqref{eq:VHatDef} we have \[ \hat{V}_{I+1}(x) = \max \left\{ \hat{V}_I(x), g_{I+1}(x) \right\} \, .\] If strong duality holds between problems \eqref{eq:OSP} and \eqref{eq:OSPD}, i.e., $J_P(\hat{x}) = J_D(\hat{x})$, it follows from \eqref{eq:gIplus1Alt} that $J_P(\hat{x}) = \mathcal{T}\hat{V}_I(\hat{x}) = g_{I+1}(\hat{x})$. Since we suppose $\hat{V}_I(\hat{x}) < \mathcal{T}\hat{V}_I(\hat{x})$, we will have $g_{I+1}(\hat{x}) > \hat{V}_I(\hat{x})$, and hence $\hat{V}_{I+1}(\hat{x}) > \hat{V}_I(\hat{x})$. Equation \eqref{eq:LemStrictIncr} follows immediately. 
\end{proof}

We use Lemmas \ref{lem:NewLB} to \ref{lem:VFIncrease} to prove two convergence results for Algorithm \ref{alg:GDDP}.

\begin{thm}[Pointwise convergence of $\hat{V}_I(x)$ as $I \rightarrow \infty$] \label{thm:VLimEverywhere}
For each $x \in \mathcal{X}$ for which $V^\star(x)$ is finite, there exists a limiting value $\hat{V}_{\rm lim}(x) \leq V^\star(x)$ such that $\lim_{I \rightarrow \infty} \hat{V}_I(x) = \hat{V}_{\rm lim}(x)$.
\end{thm}
\begin{proof}

From Lemma \ref{lem:NewLB} and \eqref{eq:VHatDef}, we have that for any $x$ with finite optimal value, the sequence $\{\hat{V}_I(x)\}_{I=0}^\infty$ is bounded from above by $V^\star(x)$. Its value is non-decreasing each time a new lower-bounding function $g_{I+1}(x)$ is generated, and therefore the limit $\hat{V}_{\rm lim}(x)$ exists by the Monotone Convergence Theorem.
\end{proof}

\begin{thm}[Finite termination of Algorithm \ref{alg:GDDP}] \label{thm:EpsConv}
Suppose the following conditions are met:
\begin{enumerate}
\item[(i)] Strong duality holds for the one-stage problem \eqref{eq:OSP} with parameter $x_m$ each time it is solved, for all $x_m \in \xalgm$.
\item[(ii)] In the limit as $I \rightarrow \infty$ each $x_m \in \xalgm$ is picked by \xpicker with strictly positive probability at each iteration. 
\item[(iii)] Each $V^\star(x_m)$ is finite.
\end{enumerate}
Then Algorithm \ref{alg:GDDP} converges in a finite number of iterations with probability $1$ for any tolerance $\delta > 0$.
\end{thm}
\begin{proof}
Application of Theorem \ref{thm:VLimEverywhere} to any $x_m \in \xalgm$ shows that $\lim_{I \rightarrow \infty} \hat{V}_I(x_m)$ exists. Each time $x_m$ is picked by \xpicker\!\!\!, from Lemma \ref{lem:VFIncrease} the value function estimate at $x_m$ increases by an amount equal to $\varepsilon(x_m; \hat{V}_I)$, as long as strong duality holds in problem \eqref{eq:OSP}. 

The convergent sequence $(\hat{V}_0(x_m), \hat{V}_1(x_m), \ldots)$ satisfies necessary conditions for a Cauchy sequence, even though $x_m$ is not chosen by \xpicker at every iteration of the algorithm. Thus, let $N(x_m,\delta)$ denote the finite iteration number beyond which all differences between later elements of the sequence have magnitude less than $\delta$. By assumption, $x_m$ will, with probability $1$, be picked at some iteration $I > N(x_m,\delta)$, at which stage we will have \[\varepsilon(x_m; \hat{V}_I) = \hat{V}_{I+1}(x_m) - \hat{V}_I(x_m) = |\hat{V}_{I+1}(x_m) - \hat{V}_I(x_m)| < \delta\] where the first equality comes from Lemma \ref{lem:VFIncrease}, the second comes from Lemma \ref{lem:VlessthanTV}, and the inequality comes from the definition of a Cauchy sequence. The constant $N(x_m,\delta)$ is different but finite for each $x_m \in \xalgm$. Applying this argument to the largest such value over all points $x_m \in \xalgm$, one deduces that the termination criterion of Algorithm \ref{alg:GDDP} will be satisfied in finite iterations.
\end{proof}

\subsection{Suboptimality of GDDP output} \label{eq:SolutionProperties}
Convergence to a final value function approximation $\hat{V}$ does not imply $\hat{V}(\hat{x}) = V^\star(\hat{x})$ for any given $\hat{x}$, even if one manages to achieve $\varepsilon(\hat{x}; \hat{V}) = 0$. It is therefore desirable to be able to relate $\hat{V}$, the output of Algorithm \ref{alg:GDDP}, to the optimal value function $V^\star$. We now state a lemma which allows us to derive bounds on the difference, supported by the following definition.

For a vector $y$ reachable from $x$ (i.e., for which problem \eqref{eq:OSP} with parameter $x$ remains feasible when augmented with the constraint $x_+ = y$), let $\theta(x,y; \hat{V})$ be the increase in optimal cost when the constraint $x_+ = y$ is added, relative to the original problem \eqref{eq:OSP}. In other words, $\theta(x,y; \hat{V})$ is the (non-negative) cost of artificially constraining the successor state to be $y$ rather than letting it be chosen freely. 

\begin{lem} \label{lem:BellmanEps}
The optimal value function $V^\star$ and approximate value function $\hat{V}$ are related by the following inequality, for any vector pair $(x, y)$ where $y$ is reachable from $x$:
\begin{equation} \label{eq:VstarVhatLemma}
V^\star(x) - \gamma V^\star(y) \leq \hat{V}(x) - \gamma \hat{V}(y) + \theta(x,y; \hat{V}) + \varepsilon(x; \hat{V})
\end{equation}
\end{lem}
\begin{proof}
Let $u^\star(x; \hat{V})$ be a minimizer in \eqref{eq:PiDef} for value function $\hat{V}$, and let $u_{x \rightarrow y}$ be an input that brings about successor state $y$ at minimum stage cost. From definition \eqref{eq:EpsDef} we have
\begin{subequations}
\allowdisplaybreaks
\begin{align}
\hat{V}(x) & = \mathcal{T}\hat{V}(x) - \varepsilon(x; \hat{V}) \\
& = \ell(x, u^\star(x; \hat{V})) + \gamma \hat{V}(f(x,u^\star(x; \hat{V})) - \varepsilon(x; \hat{V}) \\
& = \ell(x, u_{x \rightarrow y}) + \gamma \hat{V}(y) - \varepsilon(x; \hat{V}) \nonumber \\ 
& \quad + \ell(x, u^\star(x; \hat{V})) + \gamma \hat{V}(f(x,u^\star(x; \hat{V})) \nonumber \\ 
& \quad - \left(\ell(x, u_{x \rightarrow y}) + \gamma \hat{V}(y)\right) \\
& = \ell(x, u_{x \rightarrow y}) + \gamma \hat{V}(y) - \theta(x,y; \hat{V}) -\varepsilon(x; \hat{V})\label{eq:VHatGapExpr}
\end{align}
\end{subequations}
For the optimal value function, similarly defining $u(x; V^\star)$ to be a minimizer in \eqref{eq:PiDef} for value function $V^\star$, we have
\begin{subequations}
\begin{align*}
V^\star(x) & = \mathcal{T}V^\star(x) \\
& = \ell(x, u^\star(x; V^\star)) + \gamma V^\star(f(x,u^\star(x; V^\star))) \\
& \leq \ell(x, u_{x \rightarrow y}) + \gamma V^\star(y) \, ,
\end{align*}
\end{subequations}
where the last line follows from suboptimality of $u_{x \rightarrow y}$ and $y$ in the one-stage problem. Hence $\ell(x, u_{x \rightarrow y}) \geq V^\star(x) - \gamma  V^\star(y)$. Substitution into equation \eqref{eq:VHatGapExpr} completes the proof.
\end{proof}

\begin{cor} \label{cor:BellmanSeq}
Suppose there exists a state $\overline{y}$ for which it is known that $\hat{V}(\overline{y}) = V^\star(\overline{y})$. Then from inequality \eqref{eq:VstarVhatLemma},
\begin{equation}
V^\star(x) \leq \hat{V}(x) + \theta(x,\overline{y}; \hat{V}) + \varepsilon(x; \hat{V}) \, .
\end{equation}
Now suppose there exists a feasible state trajectory $(x_0, x_1, \ldots, x_T)$ where $x_T = \overline{y}$. Then, applying Lemma \ref{lem:BellmanEps} recursively from any step $t < T$:
{\allowdisplaybreaks
\begin{align*}
V^\star(x_t) & \leq \gamma V^\star(x_{t+1}) + \hat{V}(x_t) - \gamma \hat{V}(x_{t+1}) \\
& \hspace{2.8cm}+ \theta(x_t,x_{t+1}; \hat{V}) + \varepsilon(x_t; \hat{V}) \\
& \leq \gamma\left(\gamma V^\star(x_{t+2}) + \hat{V}(x_{t+1}) - \gamma \hat{V}(x_{t+2}) \right. \\
& \hspace{2.8cm} \left.+ \theta(x_{t+1},x_{t+2}; \hat{V}) + \varepsilon(x_{t+1}; \hat{V}) \right) \nonumber \\
& \quad + \hat{V}(x_t) - \gamma \hat{V}(x_{t+1}) + \theta(x_t,x_{t+1}; \hat{V}) + \varepsilon(x_t; \hat{V}) \\
& \leq \gamma^{T-t}V^\star(x_T) + \sum_{\tau=t}^{T-1} \gamma^{\tau-t} \left(\hat{V}(x_\tau) - \gamma \hat{V}(x_{\tau+1}) \right. \\
& \hspace{3.3cm} \left. + \theta(x_\tau,x_{\tau+1}; \hat{V}) + \varepsilon(x_\tau; \hat{V}) \right) \\
& = \gamma^{T-t}\left(V^\star(x_T) - \hat{V}(x_T)\right) + \hat{V}(x_t) \\
& \hspace{1cm} + \sum_{\tau=t}^{T-1} \gamma^{\tau-t} \left(\theta(x_\tau,x_{\tau+1}; \hat{V}) + \varepsilon(x_\tau; \hat{V}) \right)
\end{align*}
Since $V^\star(x) \geq \hat{V}(x)$ for any $x \in \mathcal{X}$, and $V^\star(x_T) - \hat{V}(x_T) = V^\star(\overline{y}) - \hat{V}(\overline{y}) = 0$, the optimal value function can be bounded from below and above:
\begin{align}\label{eq:BellmanSeqIneq}
\hspace{-0.3cm}\hat{V}(x_t) \leq V^\star(x_t) & \leq \hat{V}(x_t) + \sum_{\tau=t}^{T-1} \gamma^{\tau-t} \left(\theta(x_\tau,x_{\tau+1}; \hat{V}) \right. \nonumber \\
& \hspace{3.8cm} \left. + \, \varepsilon(x_\tau; \hat{V}) \right).\!\!
\end{align}
}
\end{cor}
An example of a suitable $\overline{y}$ is an equilibrium that can be maintained while incurring zero stage cost, so that $\hat{V}(\overline{y}) = V^\star(\overline{y}) = 0$. Note that the state trajectory used in Corollary \ref{cor:BellmanSeq} can be unrelated to \xalg used in Algorithm \ref{alg:GDDP}. This leads to the following observation.

\begin{rem}[Certifying $\hat{V}(x) = V^\star(x)$] \label{rem:CertOpt}
For any $x$, the value function estimate is exactly correct, i.e.~$\hat{V}(x_t) = V^\star(x_t)$, for each step $t$ along any state trajectory starting at $x_0 = x$, satisfying the following conditions:
\begin{enumerate}
\item The Bellman error $\varepsilon(x_t; \hat{V})$ equals zero at each step $t$; 
\item Each state $x_{t+1}$ is an optimal successor state of $x_t$ (i.e., an optimal $x_+$ in problem \eqref{eq:OSP}), and hence $\theta(x_t,x_{t+1}; \hat{V}) = 0$; 
\item The final state is $\overline{y}$. 
\end{enumerate}
\end{rem}

Unfortunately it is in general difficult to achieve this certification for an arbitrary point $x$. Even if Algorithm \ref{alg:GDDP} terminates with $||\varepsilon_I||_\infty = 0$ for the $M$ elements of $\xalgm$, these will have been generated when the algorithm was initialized, and will not in general contain the required greedy policy sequence satisfying $\theta(x_t,x_{t+1}; \hat{V}) = 0$ for each $t$. If the states in the sequence are not elements of $\xalgm$, then in general we will have $\varepsilon(x_t; \hat{V}) > 0$ as the algorithm was not tailored to those points. Moreover, optimal state trajectories for many infinite horizon problems are themselves infinitely long, in which case there does not even exist a finite-time check for the conditions in Remark \ref{rem:CertOpt}.

We therefore seek more practical interpretations of Corollary \ref{cor:BellmanSeq}. The following observation suggests that a bound can be derived either by playing forward the greedy policy, or by following a pre-constructed feasible trajectory.

\begin{rem}[Bounding suboptimality at a point] \label{rem:SuboptBound}

The following two methods bound $|V^\star(x) - \hat{V}(x)|$ for an arbitrary point $x \in \mathcal{X}$, assuming $V^\star(x)$ and $\hat{V}(x)$ are both finite and that there exists a $\overline{y} \in \xalgm$ for which it is known that $\hat{V}(\overline{y}) = V^\star(\overline{y})$.
\begin{enumerate}
\item[M1.] Generate a sequence $(x_0, x_1, \ldots)$ by iteratively applying the greedy policy \eqref{eq:PiDef} from an initial state $x_0 = x$. Suppose that at some time $T-1$ the state $x_{T-1}$ is in a neighbourhood of $\overline{y}$, whence there exists a feasible input that brings about $x_T = \overline{y}$. Then, for each step $\tau$ except $\tau = T-1$, we have $\theta(x_\tau,x_{\tau+1}; \hat{V}) = 0$, since $x_{\tau+1}$ is the optimal successor state of $x_\tau$ under the greedy policy. Then from relationship \eqref{eq:BellmanSeqIneq},\begin{align}
\hat{V}(x) & \leq V^\star(x) \leq \hat{V}(x) + \gamma^{T-1}\theta(x_{T-1},\overline{y}; \hat{V}) \nonumber \\
& \hspace{3.4cm}+ \sum_{\tau=0}^{T-1} \gamma^\tau  \varepsilon(x_\tau; \hat{V}) \, .
\end{align}

\item[M2.] Suppose a feasible trajectory $(x_0, x_1, \ldots, x_T)$ can be constructed where $x_0 = x$, $x_T = \overline{y}$, and $x_1, \ldots, x_{T-1}$ are elements of $\xalgm$. Then in the limit where Algorithm \ref{alg:GDDP} terminates with $\delta = 0$, we have $\varepsilon(x_\tau; \hat{V}) = 0$ for each $x_\tau \in \xalgm$, and only $x_0$, which in general is not an element of $\xalgm$, will have $\varepsilon(x_0; \hat{V}) > 0$. Then from relationship \eqref{eq:BellmanSeqIneq},
\begin{align}
\hat{V}(x) & \leq V^\star(x) \leq \hat{V}(x) + \varepsilon(x; \hat{V}) \nonumber \\
& \hspace{2.4cm} + \sum_{\tau=0}^{T-1} \gamma^\tau \theta(x_\tau,x_{\tau + 1}; \hat{V}) \, .
\end{align}
\end{enumerate}
\end{rem}
Methods M1 and M2 are two extreme varieties of bound suggested by \eqref{eq:BellmanSeqIneq}. In M1, contributions to suboptimality arise primarily from the Bellman errors $\varepsilon(x_\tau; \hat{V})$ along the ``greedy policy trajectory''. This is in effect the common practice of obtaining an upper bound by simulating the policy and measuring the incurred costs \cite{pereira_multi-stage_1991}, but with more precise treatment of the tail end of the trajectory. In M2, which could be appropriate in a motion planning or reachability setting, contributions arise primarily from the terms $\theta(x_\tau,x_{\tau + 1}; \hat{V})$. These terms reflect the added cost of forcing a state transition from $x_\tau$ to $x_{\tau+1}$ in preference to the greedy policy.

In practice a mixture of these methods can be used to construct a bound (see footnote \ref{fn:Reachability} in \S \ref{sec:Results}). 



\begin{rem}[Alternative convergence criterion]
The Bellman error convergence criterion in Algorithm \ref{alg:GDDP} could be replaced by one base on the upper bounding procedures described in Remark \ref{rem:SuboptBound}. In particular, the upper bound in M1, obtained by using $\hat{V}_I$ to simulate the evolution of the system forward from each of the states $x \in \xalgm$ is conventional in other DP and DDP literature \cite{pereira_multi-stage_1991, beuchat_point-wise_2017}. However, it is more costly to compute than the one-stage Bellman error, and does not directly inherit the termination guarantee provided by Theorem \ref{thm:EpsConv}.
\end{rem}

\subsection{Expressiveness of lower bounds $g_i(x)$} \label{sec:BoundExpr}

Algorithm \ref{alg:GDDP} produces lower bounding functions only of the specific form \eqref{eq:gIplus1}, and it may not be possible to represent $V^\star(x)$ as their pointwise maximum. One may therefore ask,
\begin{enumerate}
\item[(i)] What is the lowest number of points $M = |\xalgm|$ required to obtain $|V^\star(x) - \hat{V}(x)| \leq \overline{\delta}$ for all $x$ in a compact subset of $\mathcal{X}$, for some $\overline{\delta} > 0$?
\item[(ii)] What is the minimum number of iterations required to achieve this, when these $M$ points are used in the algorithm?
\end{enumerate}

Consider the case of LQR, where $V^\star(x) = \tfrac{1}{2}x^\top P x$ (see \S \ref{sec:LQR}). Inspection of \eqref{eq:gIplus1} shows that the lower bounds are of the form $g_i(x) = \tfrac{1}{2}x^\top Q x + p_i^\top x + s_i$ for some $p_i \in \mathbb{R}^n$ and $s_i \in \mathbb{R}$, where $Q$ is the quadratic state cost parameter. It is also straightforward to show that $P - Q \succ 0$, i.e., that $V^\star(x)$ has strictly higher curvature in all directions than the functions $g_i(x)$. 

To meet the condition in (i), one must generate enough lower-bounding functions for $\max_i g_i(x)$ to stay within $\overline{\delta}$ of $V^\star(x)$ everywhere, regardless of the performance of the algorithm. It becomes clear that this is equivalent to covering the state space with ellipsoids characterized by matrix $P-Q$. The number required is generally exponential in the state dimension. Answering (ii), since each of the lower bounds arises from a GDDP iteration, the number of iterations required is also exponential. Hence, GDDP does not overcome this aspect of the well-known curse of dimensionality \cite{chow_optimal_1991}.




\section{Implementation} \label{sec:Implementation}

We now discuss how the several degrees of freedom offered by Algorithm \ref{alg:GDDP} might be treated. 

\subsection{Choice of set \xalg} \label{sec:ChoiceXBeq}

Algorithm \ref{alg:GDDP} is agnostic to how the set \xalg is generated. One natural choice is to sample the points independently from a performance (or state relevance) weighting $\mathcal{P}_0$, which defines states $x$ where one wishes to minimize the parametric cost of problem \eqref{eq:IHProblem}. The rationale for this is that by running Algorithm \ref{alg:GDDP} on these points, one obtains a small Bellman error there; one might then assume from the Bellman optimality condition \eqref{eq:BellOpt} that this leads to a good approximation of $V^\star$ at the same points.

However, it may be beneficial to generate \xalg in other ways. For example, if the priority is to ensure that the suboptimality of the value function can always be upper-bounded as in method M2 above, states could be generated systematically using a reachability criterion or planned trajectory. 

\subsection{Sampling from $\xalgm$} \label{sec:XPickerChoices}

The function \xpicker chooses which element of \xalg to use to derive the next bound $g_{I+1}(x)$. For example:
\begin{enumerate}
\item Return a random (equiprobable or with strictly positive weights) element of \xalg at each iteration. 
\item Loop sequentially through all elements of \xalg in order, returning to $x_1$ after finishing an iteration with $\hat{x} = x_M$.
\item Select the element with the largest Bellman error. 
\item Iterate on the same $\hat{x}$ until its Bellman error $\varepsilon(\hat{x}; \hat{V}_I)$ has reduced to tolerance $\delta$. Then cycle through all points in this manner. 
\end{enumerate}
Note that strictly speaking Choices 2 and 4 additionally require \xpicker to be supplied with knowledge of previous iterations.

Choice 1 fulfils the conditions of Theorem \ref{thm:EpsConv}, and Choice 2 also converges under a trivial adaptation of the same theorem. Choices 3 and 4 are clearly inappropriate if the strong duality condition of Theorem \ref{thm:EpsConv} is not satisfied, as this will cause a persistent Bellman error. Results for Choices 1 and 3 are reported in Section \ref{sec:Results}.

\subsection{Convergence check} \label{sec:ConvCheck}

Measurement of the Bellman error in lines \ref{algl:BellGapStart}-\ref{algl:BellGapEnd} need not be performed at every iteration $I$.  A measurement for all elements of \xalg requires $M$ solutions of problem \eqref{eq:OSP} or \eqref{eq:OSPD}, each of which could just as well be used to generate a new lower-bounding function. It may be more attractive to spend a greater share of computation time creating lower-bounding functions, and check convergence less frequently.

\subsection{Convexity and solver compatibility}

Because Algorithm \ref{alg:GDDP} relies on repeated solution of problem \eqref{eq:OSP} or its dual \eqref{eq:OSPD}, one may wish to minimize solution time by considering convexity, problem class, or the accumulation of lower-bounding functions.

\subsubsection{Convexity of problem \eqref{eq:OSP}}
On initialization of the algorithm with $g_0(x) = 0$, the one-stage problem has a convex objective and constraints for any parameter $\hat{x}$ if $R_j \succeq 0$ for all constraints $j$ in \eqref{eq:OSPStageCost}. Therefore the only potential source of non-convexity as the algorithm progresses is the introduction of a non-convex lower bound $g_i(x)$. Inspection of \eqref{eq:gIplus1} indicates that to preserve convexity it is sufficient, but not necessary, for all of $\phi(x)^\top \hat{\lambda}^\star_\beta$, $f_x(x)^\top \hat{\nu}^\star$, and $h(x)^\top \hat{\lambda}^\star_c$ to be convex. This holds in many cases, e.g.~$h(x)$ and $f_x(x)$ affine and $\phi(x)$ convex. Even if $g_{I+1}(x)$ is not globally convex, it may still be convex in the region where it is ``active'', i.e.~where $g_{I+1}(x) = \max_i g_i(x) = \hat{V}_{I+1}(x)$. Lastly, one may choose not to add a lower-bounding function that fails a convexity test.

\subsubsection{Simplified lower-bounding functions}
If in line \ref{algl:AddLB} the new lower bound \eqref{eq:gIplus1} can be approximated from below by  its first- or second-order Taylor expansion around $\hat{x}$, the simpler functions may usefully restrict problem \eqref{eq:OSP} or its dual \eqref{eq:OSPD} to a class for which more efficient solvers exist, for example a linear or quadratically-constrained program. These simpler functions will however be less tight a lower bound on $V^\star(x)$.

\subsubsection{Redundant lower-bounding functions}
The GDDP algorithm adds a constraint \eqref{eq:OSPLB} to the one-stage problem at each iteration, which may make an existing one redundant, i.e.~$g_i(x) \leq \max_{j \neq i}g_j(x) \,\, \forall x \in \mathcal{X}$. As in conventional DDP, it may be desirable to ``prune'' these redundant lower-bounding functions \cite{de_matos_improving_2015} where they can be identified efficiently.


\section{Numerical results} \label{sec:Results}

We now present simulated results for constrained linear systems of various sizes, as well as a nonlinear example.

\subsection{Random linear systems} \label{sec:RandLinSys}
Random asymptotically stable, controllable linear systems (controllable pairs $(A,B)$ with $\sigma(A) \leq 0.99$) were created for different state and input dimensions $n$ and $m$. The state and input spaces were $\mathcal{X} = \mathbb{R}^n$ and $\mathcal{U} = \mathbb{R}^m$ respectively. Constraint \eqref{eq:IHSIC} was used to bound the input with $||u||_\infty \leq 1$, the stage cost was $\ell(x,u) = \tfrac{1}{2}x^\top x + \tfrac{1}{2}u^\top u$, and the discount factor was $\gamma=1$. 

The elements of \xalg were drawn from a normal distribution $\mathcal{P}_0 = \mathcal{N}(0, 5^2 \cdot I_n)$. This distribution can be viewed as a weighting of the states for which we wish to solve problem \eqref{eq:IHProblem}; see Section \ref{sec:ChoiceXBeq}.

For these systems, each lower-bounding function $g_i(x)$ generated by GDDP is a convex quadratic, and the one-stage problems \eqref{eq:OSP} are quadratically-constrained linear programs. These were solved in Gurobi 7.0.2, on a computer with a 2.6 GHz Intel Core i7 processor and 16 GB of RAM.

\subsubsection{Iterations to termination}
For a given dimension $(n,m)$, the algorithm was run for an illustrative random system to a tolerance of $\delta = 10^{-3}$. The \xpicker function used Choice 3 described in Section \ref{sec:XPickerChoices}.\footnote{Although this requires a relatively costly measurement of all Bellman errors at each iteration, it arguably gives a fairer illustration of the best-case number of iterations required than a random choice of $x_m$.} Table \ref{tab:GDDPIterations} shows the iterations required as a function of $M$ in each instance. Although the computational cost of solving the one-stage problem increases with system size, the number of iterations appears to be roughly linear in $M$ and unrelated to $n$ and $m$.

\begin{table}[tp]
\begin{center}
\caption{GDDP iterations, $\delta = 10^{-3}$, single system instance}
\begin{tabular}{|c|c|c|c|c|c|c|c|c|}
\hline
 States   &  Inputs   & \multicolumn{7}{c|}{$M$ (number of elements in $\xalgm$)}\\
$n$ & $m$ & 1 & 2 & 5 & 10 & 20 & 50 & 100 \\
\hline
\hline
 1 & 1 & 2 & 2 & 6 & 20 & 40 & 110 & 180 \\
\hline
 2 & 1 & 2 & 3 & 7 & 13 & 45 & 123 & 269 \\
\hline
 3 & 1 & 3 & 6 & 13 & 30 & 65 & 160 & 354 \\
\hline
 4 & 2 & 2 & 3 & 6  & 13 & 26 & 66 & 130 \\
\hline
 5 & 2 & 3 & 4 & 14 & 32 & 56 & 167 & 326 \\
\hline 
 8 & 3 & 2 & 3 & 6 & 16 & 41 & 115 & 263 \\
\hline
10 & 4 & 2 & 3 & 8 & 14 & 36 & 100 & 213 \\
\hline 
\end{tabular}\label{tab:GDDPIterations}
\end{center}

\end{table}%

\subsubsection{Approximation quality}
The quality of the value function $\hat{V}$ output by the GDDP algorithm can be measured by taking an expectation of the infinite-horizon cost over initial states $x_0 \sim \mathcal{P}_0$. For any sample $x_0$, we used inequality \eqref{eq:BellmanSeqIneq} to bound the suboptimality of this infinite-horizon cost with respect to the solution of \eqref{eq:IHProblem}.\footnote{In these tests, we simulated the system for 30 time steps from each sampled starting state (50 time steps for the last two rows). To avoid any kind of truncation effect at the end of the horizon, method M1 in Remark \ref{rem:SuboptBound} was used until $k$ steps before the end of the horizon, where $k \leq n$ is the number of steps needed to form a full-rank matrix $[B, AB, \ldots, A^{k-1}B]$. An input sequence, which (due to proximity to the origin) was in all cases feasible with respect to the box constraint on the input, was then generated for the last $k$ steps, in the sense of the reachability argument made for method M2. \label{fn:Reachability}}

Table \ref{tab:GDDPResults} compares Bellman errors and the suboptimality bounds in percentage terms for 
\begin{itemize}
\item[(a)] Set $\xalgm$, consisting of $M=200$ points drawn from $\mathcal{P}_0$.
\item[(b)] An ``out-of-sample'' evaluation set of 1,000 points drawn independently from the same distribution. 
\end{itemize}
Choice 1 was used for \xpicker$\!\!$, i.e.~an equiprobable random choice of $x_m$ at each iteration. The quantities reported were measured after 200 iterations, such that each element of $\xalgm$ was visited once on average. Each row of the table reports mean values over 20 random systems.

\begin{table*}[t!bp]
\begin{center}
\caption{GDDP solution quality after 200 iterations; averages across 20 random linear systems per row}
\begin{tabular}{|c|c|c|c|c|c|c|c|c|c|}
\hline
 States   &  Inputs   & \multicolumn{2}{c|}{$M = 200$ elements of \xalg} & \multicolumn{2}{c|}{1,000 independent samples from $\mathcal{P}_0$} & $\mathcal{P}_0$/\xalg & GDDP & Val.~it. & MPT \\
$n$ & $m$ &  RBE (\%)${}^1$ & Subopt.~bd.~(\%)${}^2$  & RBE (\%)${}^1$ & Subopt.~bd.~(\%)${}^2$ & subopt. ratio & time (s)${}^3$ & time (s) & time (s)\\
\hline
\hline
1 & 1 & 0.04038 & 0.3197 & 0.05931 & 0.3584 & 1.12 & 17.06 & 0.8652 & 0.8943 \\
\hline
2 & 1 & 0.3958 & 2.104 & 0.5860 & 2.666 & 1.27 & 19.10 & 113.3 & 3.185 \\
\hline
3 & 1 & 1.160 & 6.102 & 1.807 & 7.324 & 1.20 & 19.34 & 12,890 & 19.20 \\
\hline
4 & 2 & 1.080 & 4.357 & 2.068 & 5.934 & 1.36 & 23.68 & ---${}^4$ & 1989 \\
\hline
5 & 2 & 2.095 & 7.894 & 4.192 & 10.61 & 1.34 & 24.70 & ---${}^4$ & ---${}^4$ \\
\hline
8 & 3 & 4.695 & 12.20 & 10.45 & 17.49 & 1.43 & 34.83 & ---${}^4$ & ---${}^4$ \\
\hline
10 & 4 & 5.386 & 11.96 & 12.76 & 17.53 & 1.47 & 47.14 & ---${}^4$ & ---${}^4$ \\
\hline
\end{tabular}\label{tab:GDDPResults}
~\\
${}^1$Relative Bellman error, $(\mathcal{T}\hat{V}(x) - \hat{V}(x))/\hat{V}(x)$, averaged across all samples and systems. ${}^2$Suboptimality bound obtained via closed-loop simulation, weighted by $\hat{V}(x)$. ${}^3$Mean total time taken to generate all lower bounds $g_i(x)$ for systems in each row. ${}^4$Mean time above 6 hours or out of memory.
\end{center}
\end{table*}%

The algorithm fits a lower-bounding function at the points in $\xalgm$, and as a result the suboptimality bounds computed are a little worse at the out-of-sample points (17.5\% vs.~12.0\% for the largest system studied). Fig.~\ref{fig:5s2i} shows a representative plot of convergence of the suboptimality bound in the case of a 5-state, 2-input system.

\begin{figure}[tp]
\begin{center}
\includegraphics[width=8.5cm]{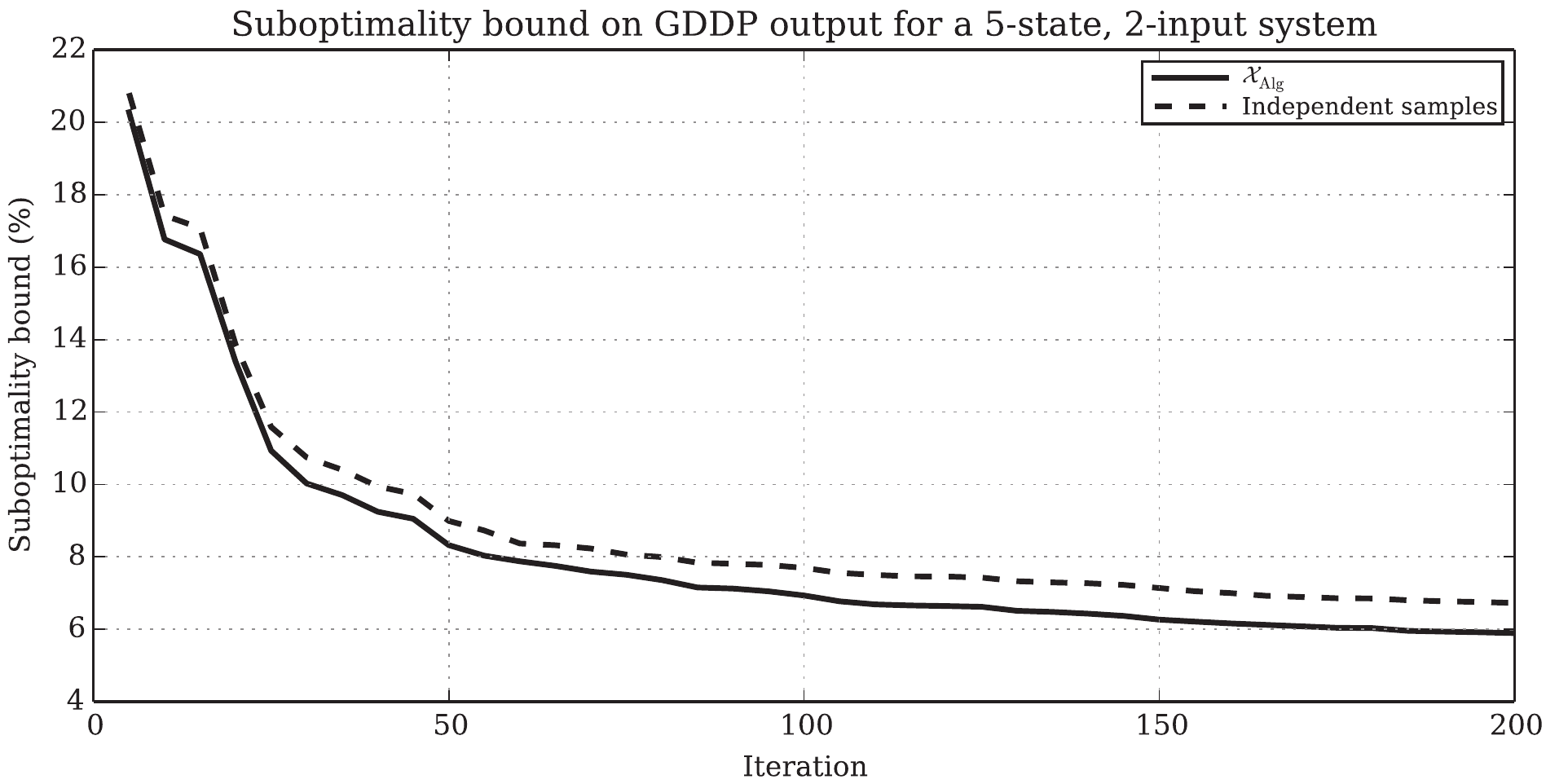}
\end{center}
\caption{Representative convergence behaviour of GDDP for a randomly-generated 5-state, 2-input system, in terms of the suboptimality implied by inequality \eqref{eq:BellmanSeqIneq} measured from closed-loop simulations. The two lines are $\hat{V}(x)$-weighted averages for the $M=200$ elements of \xalg and the 1000 samples generated independently from $\mathcal{P}_0$. Convergence was evaluated every 5 iterations.}
\label{fig:5s2i}
\end{figure}

In addition, Table \ref{tab:GDDPResults} includes comparisons with two other DP methods.

\subsubsection*{Value iteration}

Standard repeated application of the Bellman operator \cite[\S 1.2]{bertsekas_dynamic_2012}, was used to derive gridded value function approximations over the compact region $-10 \cdot \mathbf{1} \leq x \leq 10 \cdot \mathbf{1}$, i.e., two standard deviations of $\mathcal{P}_0$ from the origin in each coordinate. The state was discretized to 50 equal intervals (51 grid points) in each coordinate, and the input was discretized to 10 equal intervals (11 grid points) in each coordinate. The iterations were terminated when the maximum absolute change in value fell below $10^{-3}$. Average total times are reported in the penultimate column of Table \ref{tab:GDDPResults}. The computation time grew dramatically with system size, and while a more efficient implementation would likely reduce the times shown, the poor scaling behaviour of this approach is well known and would remain essentially unchanged.

\subsubsection*{Parametric solution}

For systems of this class, an explicit representation of the optimal value function can be obtained using the Multi-Parametric Toolbox \cite{herceg_multi-parametric_2013} for a given finite optimization horizon. As an illustrative comparison, we provide the mean computation time required to generate this optimal value function for a horizon of 10 steps. We note that the region in which the explicit value function is calculated in this manner overlaps only partially with the compact region used for value iteration, and that calculating the number of steps required to cover the same region entirely is itself a computationally expensive exercise. Average times are given in the last column of Table \ref{tab:GDDPResults}. As with value iteration, the computation time increased dramatically with system dimension, although we note that for systems small enough for the explicit controller to be computed, the online effort to evaluate the resulting policy \cite{jones_logarithmic-time_2006} would be substantially lower than solving \eqref{eq:PiDef} with the approximate value function returned by GDDP.

\subsection{Nonlinear system}

We demonstrate GDDP for the 4-state simplified ball-and-beam example from \cite[\S 10.2]{sastry_nonlinear_1999}. The state vector is $x = [r, \dot{r}, \theta, \dot{\theta}]^\top$ where $r$ is the position in metres along the beam from the pivot, and $\theta$ is the angle in radians from horizontal measured counter-clockwise. The simplification we adopt from \cite{sastry_nonlinear_1999} is to model the rolling of the ball as frictionless sliding, to avoid having to include rolling and contact interactions between the ball and beam.

The single input is a torque $\tau$ in Newtons applied at the pivot, constrained to an interval $-\tau_{\rm max} \leq u \leq \tau_{\rm max}$. The Euler-discretized dynamics are $x_+ = f_x(x) + F_u(x)u$, with
\begin{equation}
f_x(x) = x + \bmat{\dot{r} \\ r \dot{\theta}^2 - g \sin \theta \\ \dot{\theta} \\ -\frac{2mr\dot{r} + mgr\cos\theta}{mr^2 + J_b} } \!\!\Delta t \, , \,\, F_u(x) = \!\bmat{0 \\ 0 \\ 0 \\ \frac{1}{mr^2 + J_b}} \!\!\Delta t \, ,
\end{equation}
where $m$ is the mass of the ball, $J_b$ is the moment of inertia of the beam, $g$ is gravitational acceleration, and $\Delta t$ is the time discretization interval.

The parameters used were $m = 0.1$ kg, $J_b = 0.5$ kg m${}^2$, $g = 9.81$ m s${}^{-2}$, and $\Delta t = 0.1$ s. The cost function was $\ell(x,u) = \tfrac{1}{2}x^\top Q x + \tfrac{1}{2}u^\top Ru$, with $Q = {\rm diag}\{10, 1, 1, 1\}$ and $R = 0.01$, and the input constraint was $\tau_{\rm max} = 3$. The set \xalg contained $M=500$ samples, half of which were drawn from a normal distribution $\mathcal{N}(0, \Xi_1)$ with covariance matrix $\Xi_1 = {\rm diag}\{0.5^2, 0.5^2, 0.5^2, 0.5^2\}$, and the other half of which were samples from $\mathcal{N}(0, \Xi_2)$ with $\Xi_2 = {\rm diag}\{0.1^2, 0.1^2, 0.1^2, 0.1^2\}$. This choice attached relatively high weight to behaviour around the equilibrium position. The one-stage problems were solved by brute force, with no special adaptations to account for the problem structure.
 
Fig.~\ref{fig:bb_transient} shows regulation of the system to the origin from $x_0 = [1, 0, -0.1745, 0]^\top$ (the ball 1 m to the right of the pivot in a position 10 degrees below horizontal) under the derived control policy. Results are plotted for 100, 200, 300, and 400 GDDP iterations, which took 299, 723, 1206, and 1705 seconds respectively. The control policy after 100 iterations caused the state trajectory to diverge, whereas the control was stabilizing with improving transient performance after higher numbers of iterations. 

A linear feedback controller based on the Riccati solution for the system linearized about the origin (see \S \ref{sec:LQR}) was not able to stabilize the system from this initial condition. 

\begin{figure}[tp] 
\begin{center}
\includegraphics[width=8.6cm]{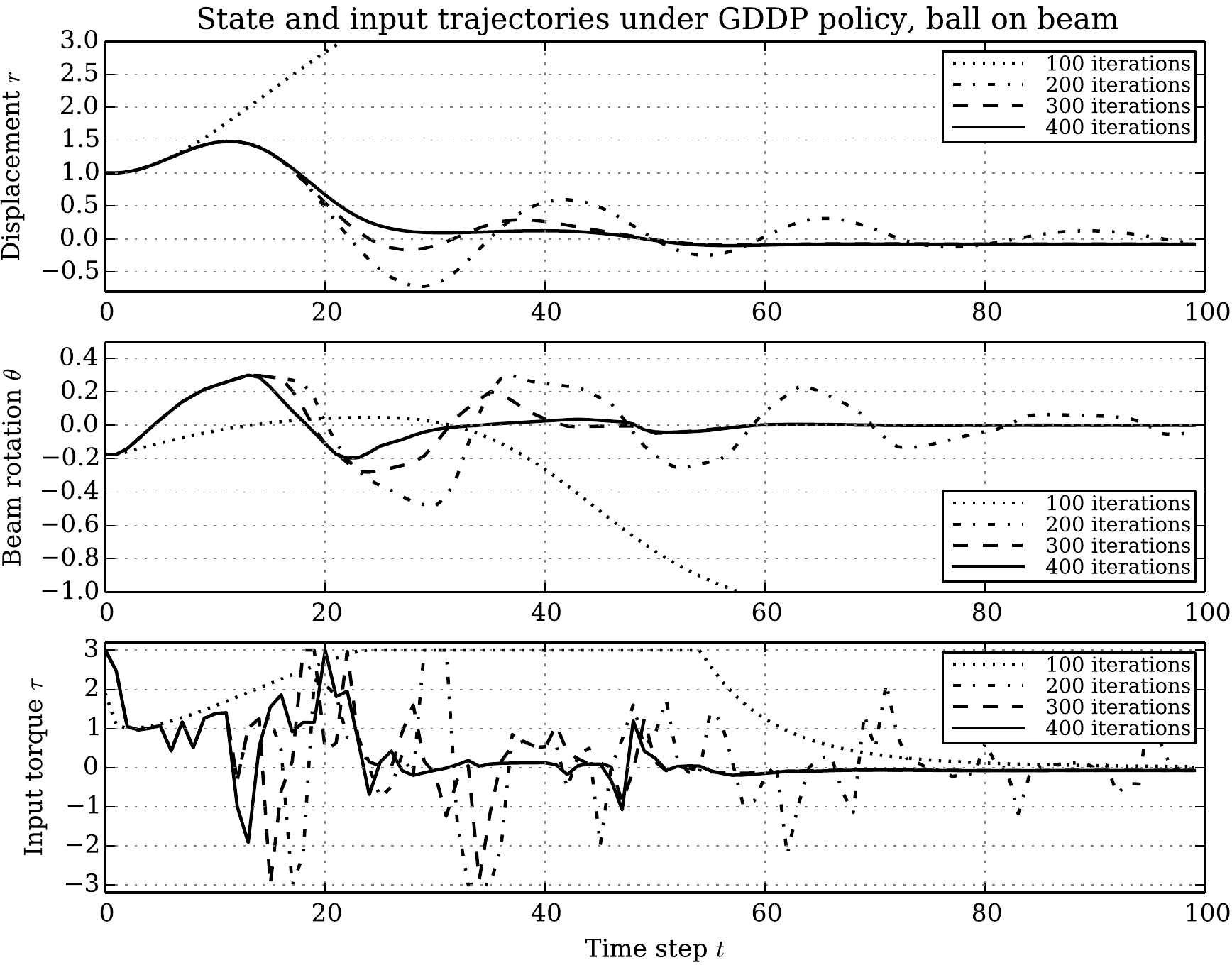}
\end{center}
\vspace{-0.5cm}
\caption{State and input trajectory from $x_0 = [1, 0, -0.1745, 0]^\top$ under the greedy policy obtained after 100, 200, 300, and 400 iterations. Small residual state offsets after the transient are due to the inexact value function approximation around the origin.}
\label{fig:bb_transient}
\end{figure}


\section{Conclusion} \label{sec:Conclusion}

This article proposed a means of constructing a series of lower bounding functions whose pointwise maximum achieves progressively tighter approximations of the optimal value function for an infinite-horizon control problem. In the linear examples tested, good closed-loop bounds on performance were achieved in only a few dozen iterations.

A number of potential extensions present themselves. Firstly, a stronger connection to the finite-horizon DDP algorithm could be made by considering sequences of states, for example by solving a multi-stage problem in place of the one-stage problem at each iteration of the algorithm. In a finite horizon setting, the conventional DDP algorithm \cite{pereira_multi-stage_1991} works by refining such a sequence from a known initial state until an upper and lower bound on the optimal trajectory cost have converged, whereas in the GDDP algorithm presented here, the successor states $x_+$ are not used as sample states for the construction of subsequent lower-bounding functions. The notion of refining policies along state trajectories is well-known in the reinforcement learning literature, where on- and off-policy learning are often contrasted \cite{busoniu_reinforcement_2010}. It may be that a mixture of the two can achieve improvements in our setting.

Secondly, it would be attractive to be able to extend the class of lower-bounding functions present in definition \eqref{eq:gIplus1} to be effective for a wider range of problems with highly non-convex value functions, for example those encountered in complex reinforcement learning problems. Presently, a strict increase in the value function lower bound is only guaranteed at a state $\hat{x}$ if strong duality holds in the one-stage problem \eqref{eq:OSP}. However, this problem and its dual can be defined in numerous ways starting from the original statement \eqref{eq:IHProblem}, and it is not clear whether an alternative formulation might have attractions over the one we have presented.

Thirdly, although we have proven finite convergence of the algorithm (to a desired Bellman tolerance) in the case of strong duality in each one-stage problem instance, we do not derive \textit{a priori} bounds on the number of iterations, or the cardinality of $\xalgm$, required to achieve a desired value function approximation accuracy in the sense of the discussion in Section \ref{sec:BoundExpr}. These appear to be difficult problems whose solutions rely on characterizing the accumulating lower-bounding constraints as a function of the original problem data. 

\bibliographystyle{plain}
\bibliography{warrington}
\begin{IEEEbiography}[{\includegraphics[width=1in,height=1.25in,clip,keepaspectratio]{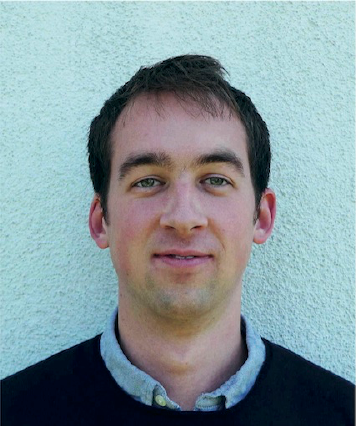}}]{Joseph Warrington}
is a Senior Scientist in the Automatic Control Lab (IfA) at ETH Zurich, Switzerland. His Ph.D.~is from ETH Zurich (2013), and his B.A.~and M.Eng.~degrees in Mechanical Engineering are from the University of Cambridge (2008). From 2014-2016 he worked as an energy consultant at Baringa Partners LLP, London, UK, and he has also worked as a control systems engineer at Wind Technologies Ltd., Cambridge, UK, and privately as an operations research consultant. He is the recipient of the 2015 ABB Research Prize for an outstanding PhD thesis in automation and control, and a Simons-Berkeley Fellowship for the period January-May 2018. His research interests include dynamic programming, large-scale optimization, and predictive control, with applications including power systems and transportation networks.
\end{IEEEbiography}
\vfill
\newpage
\begin{IEEEbiography}[{\includegraphics[width=1in,height=1.25in,clip,keepaspectratio]{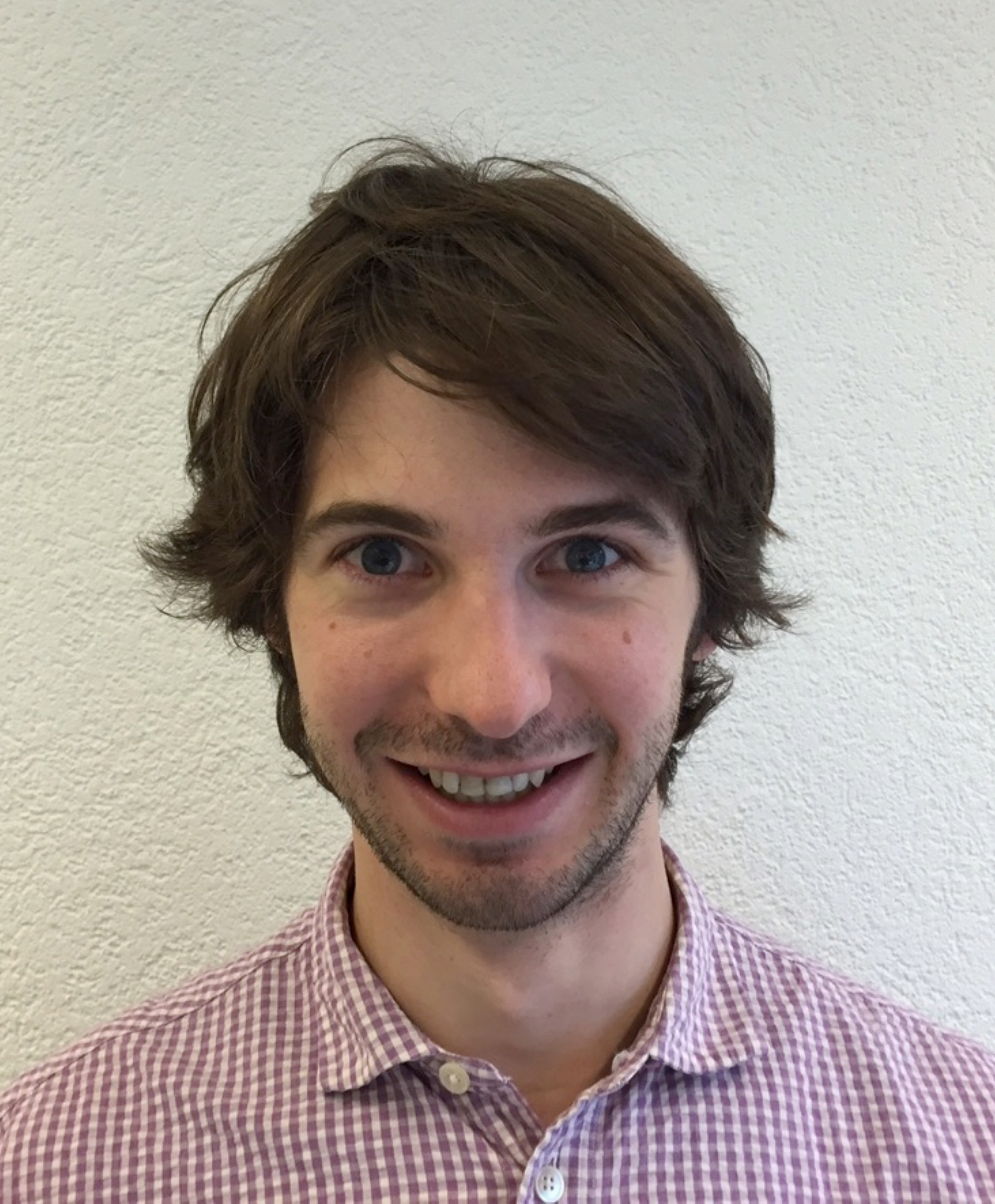}}]{Paul N. Beuchat}
received the B.Eng.~degree in mechanical engineering and B.Sc. in physics from the University of Melbourne, Australia, in 2008, and the M.Sc.~degree in robotics, systems and control from ETH Zurich, Switzerland, in 2014, where he is currently working towards the Ph.D. degree at the Automatic Control Laboratory. From 2009-2012 he worked as a subsurface engineer for ExxonMobil, based in Melbourne, Australia. His research interests are control and optimization of large scale systems, with a focus towards developing approximate dynamic programming techniques for applications in the areas of power systems, building control, and coordinated flight.
\end{IEEEbiography}
\begin{IEEEbiography}[{\includegraphics[width=1in,height=1.25in,clip,keepaspectratio]{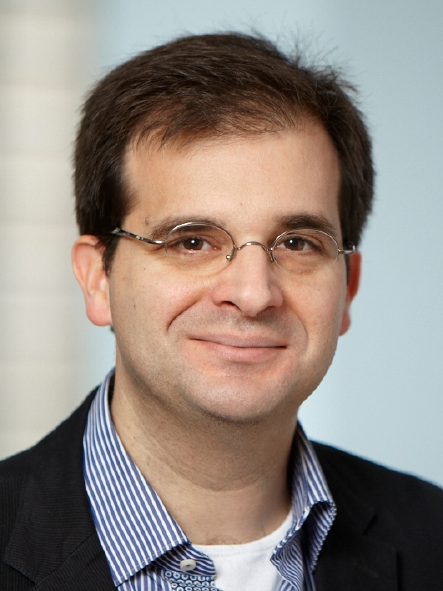}}]{John Lygeros}
completed a B.Eng. degree in electrical engineering in  1990 and an M.Sc. degree in Systems Control in 1991, both at Imperial College of Science Technology and Medicine, London, UK. In 1996 he obtained a Ph.D.~degree from the Electrical Engineering and Computer Sciences Department, University of California, Berkeley. During the period 1996--2000 he held a series of post-doctoral researcher appointments at the Laboratory for Computer Science, M.I.T., and the Electrical Engineering and Computer Sciences Department at U.C.~Berkeley. Between 2000 and 2003 he was a University Lecturer at the Department of Engineering, University of Cambridge, UK, and a Fellow of Churchill College. Between 2003 and 2006 he was an Assistant Professor at the Department of Electrical and Computer Engineering, University of Patras, Greece. In July 2006 he joined the Automatic Control Laboratory at ETH Zurich, where he is currently serving as the Head of the Automatic Control Laboratory and the Head of the Department of Information Technology and Electrical Engineering. His research interests include modelling, analysis, and control of hierarchical, hybrid, and stochastic systems, with applications to biochemical networks, automated highway systems, air traffic management, power grids and camera networks. John Lygeros is a Fellow of the IEEE, and a member of the IET and the Technical Chamber of Greece; since 2013 he has served as the Treasurer of the International Federation of Automatic Control.
\end{IEEEbiography}
\vfill
\end{document}